\documentclass[12pt]{amsart}

\usepackage{
amsfonts,
latexsym,
amssymb,
mathabx,
enumerate,
verbatim,
amsthm,
mathrsfs,
}

 \changenotsign  

\newcommand{\labbel}{\label}

\newtheorem{theorem}{Theorem}[section]
\newtheorem{lemma}[theorem]{Lemma}

\newtheorem{proposition}[theorem]{Proposition} 
 
\newtheorem{corollary}[theorem]{Corollary}

\newtheorem*{theorem*}{Theorem}
\newtheorem*{corollary*}{Corollary}
\newtheorem*{proposition*}{Proposition}

\theoremstyle{definition}
\newtheorem{definition}[theorem]{Definition}

\theoremstyle{remark}
\newtheorem{remark}[theorem]{Remark}
\newtheorem*{remark*}{Remark}
\newtheorem*{definition*}{Definition}
\newtheorem*{disclaimer*}{Disclaimer}
 
\newtheorem{example}[theorem]{Example}

\newcommand{\+}{\mathbin{\hash}}
\DeclareMathOperator*{\pl}{\mathbin{\hash}}

\DeclareMathOperator*{\nsum}{\ifdim\displaywidth>0pt {{\mbox{\Large \#}}}\else{{\mbox{\large \#}}}\fi}

\begin{document}
 
\title{An Infinite Natural Sum}

\author{Paolo Lipparini} 
\address{Dipartimento Diaccio di Matematica\\Viale della Ricerca Scientifica\\II Universit\`a di Roma (Tor Vergata)\\I 00133 Rome \\ Italy}
\urladdr{http://www.mat.uniroma2.it/\textasciitilde lipparin}
\email{lipparin@axp.mat.uniroma2.it}


\thanks{Work performed under the auspices of G.N.S.A.G.A}

\begin{abstract}
As  far as
algebraic properties are concerned,
 the usual addition on the class of ordinal numbers 
is not really well behaved; for example,
it is not commutative, nor left cancellative etc.
In a few cases,
the \emph{natural 
Hessenberg 
sum} is a better alternative,
since it shares most of the usual properties
of the addition on the naturals.

A countably infinite iteration of the natural sum 
has
been  used in a recent paper by V\"a\"an\"anen and Wang,
with applications to infinitary logics.
We present a detailed study of this infinitary operation,
showing that there are many similarities with
the ordinary infinitary sum, and providing connections with 
certain kinds of   infinite mixed sums.
\end{abstract} 
 
\maketitle

\section{Introduction} \labbel{intro}

There are different ways to extend the addition operation
from the set $ \omega$ of natural numbers
to the class of ordinals.
The standard way is to take
$\alpha+ \beta $ as the ordinal which represents the 
 order type of 
$\alpha$ with a copy of $\beta$ added at the top.
This operation can be introduced by the customary 
inductive definition and satisfies only few of the familiar
properties shared by the addition on the naturals.

On the other hand, again on the class of the ordinals, one can define the
 \emph{(Hessenberg) natural sum} $\alpha\+ \beta $ 
of $\alpha$ and $\beta$ by expressing $\alpha$ and $\beta$
in Cantor normal form and ``summing linearly''. 
See below for further details.
The resulting operation $\+$ is commutative, associative and cancellative.
It can be given an inductive definition as follows.
\begin{equation}\labbel{defnat}
\begin{split}   
&0\+0=0\\
&\alpha \+ \beta =
 \sup _{ 
\substack{\alpha' < \alpha\\ \beta' < \beta}
} 
\{ S(\alpha \+ \beta'), S(\alpha' \+ \beta) \}  
\end{split}  
\end{equation}   
where $S$ denotes \emph{successor}.

It is relevant that the natural sum, too, admits an order theoretical definition.
If $\alpha$, $\beta$ and $\gamma$ are ordinals,
$\gamma$ is said to be a \emph{mixed sum}
of $\alpha$ and $\beta$ if there are disjoint subsets
$A$ and $B$ of $\gamma$ such that 
$\gamma= A \cup B$ and $A$, $B$ have order type,
respectively, $\alpha$ and $\beta$, under the order induced
by  $\gamma$.
P. W. Carruth \cite{Car} showed that
$\alpha \+ \beta $ is the largest mixed sum
of $\alpha$ and $\beta$. He also found many applications. 

In V\"a\"an\"anen and Wang \cite{VW} the authors define
 a countably infinite
extension of $\+$ by taking supremum at the limit stage.
They provide applications  to infinitary logics.
Subsequently,
we have found applications to
compactness of topological spaces
in the spirit of \cite{OC},
in particular, with respect to Frol\'\i k  sums.

Carruth Theorem, as it stands, cannot be generalized to such 
an infinite natural sum.
Indeed, 
every countably infinite ordinal
is an ``infinite mixed sum'' of countably many 
$1$'s, hence in the infinite case the maximum is not 
necessarily attained.
 See Definition \ref{infmix} and the comment 
after Theorem \ref{carr}.

However, we show that
Carruth Theorem can  indeed be generalized, provided
we restrict ourselves to certain well behaved
infinite mixed sums.
In order to provide this generalization,
we need a finer description of the
countably infinite natural sum.
We show that any infinite natural sum
can be computed in two steps: in
the first step one takes the  natural sum
of some sufficiently large \emph{finite} set
of summands.  In the second step one  adds
the infinite ordinary sum of the remaining summands.
In other words, the infinite natural sum and the 
more usual infinite sum differ only for a finite ``head''
and they agree on the remaining ``tail''. This is
used in order to show that the infinite natural sum of a
sequence is the maximum of all possible infinite mixed sums
made of elements from the sequence,
provided one restricts only to mixed sums 
satisfying an appropriate finiteness condition.

In the end, we show that the infinite natural sum can be actually
computed as some finite natural sum, and can be expressed
in terms of the Cantor normal forms of the summands.
We show that a sequence has only a finite number of
mixed sums satisfying an additional convexity property.
This extends a classical theorem by Sierpinski \cite{Sie},
asserting that one gets only a finite 
number of values for the sum of 
some fixed countable sequence of ordinals, by changing their order.

\section{Natural sums} \labbel{natsumsec} 

We now give more details 
about the definitions hinted above 
and list some simple facts about the natural sums.
Here and below sums, products and exponentiations will be always intended
in the ordinal sense.
See, e.~g., the books Bachmann \cite{Bac} and  
 Sierpinski \cite{Sier} for a detailed introduction
to ordinal operations.
Recall that every ordinal $ \alpha  >0$ can be expressed in
a unique way in \emph{Cantor normal form} as follows 
\begin{equation} \labbel{cantor}   
 \alpha  =
 \omega ^ {\xi_k} r_k + \omega ^ {\xi _{k-1}}r _{k-1}    +
\dots
+ \omega ^ {\xi_1} r_1 + \omega ^ {\xi_0}r_0  
 \end{equation}
  for 
 integers 
$k \geq 0$, $r_k, \dots, r_0 >0$ and ordinals
$ \xi _k > \xi _{k-1} > \dots > \xi_1 > \xi_0  $.

\begin{definition} \labbel{natsum}
The natural sum $\alpha \+ \beta $ of two ordinals 
$\alpha$ and $\beta$ 
is the only operation
satisfying
\begin{equation*}    
 \alpha \+ \beta  =
 \omega ^ {\xi_k} (r_k + s_k) +
\dots
+ \omega ^ {\xi_1} (r_1+s_1) + \omega ^ {\xi_0}(r_0 + s_0) 
 \end{equation*}
whenever
\begin{equation*}
\begin{split}   
 \alpha   =
 \omega ^ {\xi_k} r_k  +
\dots
+ \omega ^ {\xi_1} r_1 + \omega ^ {\xi_0}r_0\\  
 \beta  =
 \omega ^ {\xi_k}  s_k +
\dots
+ \omega ^ {\xi_1} s_1 + \omega ^ {\xi_0} s_0 
 \end{split} 
\end{equation*}
and $k , r_k, \dots, r_0, s_k, \dots, s_0  < \omega$, 
$ \xi _k > \dots > \xi_1 > \xi_0  $.
 \end{definition}   

The definition is justified by the fact that we can represent every
nonzero $\alpha$ and $\beta$ in Cantor normal form and then insert
some more null coefficients for convenience
just in order to make the indices match.
The null coefficients do not affect the  ordinals,
hence the definition is well-posed.
See, e.~g.,  \cite{Bac,Sier} for further details.
 
An elegant way to introduce the natural sum
is obtained by expressing equation
\eqref{cantor} 
 in a conventional way as     
$\alpha = \sum _{ \xi \in F} \omega ^ \xi ,$
where $F$ is the finite \emph{multiset} which contains each
$\xi_ \ell$ exactly $r _ \ell$ times.
This is justified by the fact that, say,
$\omega ^ \xi + \omega ^ \xi = \omega ^ \xi 2$.
In this way,  $\alpha= 0 $ is expressed by summing over the
empty multiset. 
Of course, when expanding the above summation,
one should be careful to consider the terms with larger
exponents first, that is, write them on the left.
If $ \alpha = \sum _{ \xi \in F} \omega ^ \xi $ and
$ \beta  = \sum _{ \xi \in T} \omega ^ \xi $,
then
$ \alpha  \+ \beta $ is defined as $ \sum _{ \xi \in F \cup T} \omega ^ \xi $,
where in the  union $F \cup T$ we take into account multiplicities.
In this note, however, we shall follow the more conventional
notations.

It can be shown 
by induction on $(\max \{ \alpha, \beta  \}, \min \{ \alpha, \beta  \} )$,
ordered lexicographically, 
that Definition \ref{natsum} 
is equivalent to the definition given by means of
equations \eqref{defnat}. This shall not be needed in what follows.
 
Notice that the assumption
$ \xi _k > \dots > \xi_1 > \xi_0  $ in Definition \ref{natsum} 
is necessary, since, for example,
$(1 + \omega) \+ (1 + \omega 0)$ is $ \omega \+ 1 = \omega + 1$,
while summing ``linearly'' we would obtain
$2+ \omega = \omega $.
However, the assumption that
$ \xi _k > \dots > \xi_1 > \xi_0  $
can be relaxed to
$ \xi _k \geq \dots \geq\xi_1 \geq \xi_0  $.

\begin{proposition} \labbel{facts}
Let $\alpha$, $\beta$ and $\eta$ be ordinals. 
  \begin{enumerate}[(1)]    
\item
The operation $\+ $ is commutative, associative, both left and right cancellative
and strictly monotone in both arguments.
\item
$ \sup \{ \alpha, \beta \} \leq  \alpha + \beta \leq \alpha \+ \beta $.
\item
If $\alpha, \beta < \omega ^ \eta $,
then  $\alpha \+ \beta < \omega ^ \eta $.
\item
If $\beta < \omega ^ \eta $,
then  $\alpha \+ \beta < \alpha  + \omega ^ \eta $.
\item
If $\beta < \omega ^ \eta $,
then
 $(\alpha \+ \beta) + \omega ^ \eta = \alpha  + \omega ^ \eta $.
   \end{enumerate}
 \end{proposition} 

 \begin{proof}
Everything is almost immediate from Definition \ref{natsum}.

For example, to prove (4), let 
$ \alpha   =
 \omega ^ {\xi_k} r_k  +
\dots
+ \omega ^ { \eta } r + \dots $
 with, as usual, the exponents of $ \omega$ 
in decreasing order, and where we can allow
$r$ to be $0$.
Then
$ \alpha + \omega ^ \eta   =
 \omega ^ {\xi_k} r_k  +
\dots
+ \omega ^ { \eta } (r+1)$,
 while, if
$\beta < \omega ^ { \eta }$, then
$ \alpha \+ \beta    =
 \omega ^ {\xi_k} r_k  +
\dots
+ \omega ^ { \eta } r + \dots$,
since $\beta$ does not contribute to 
summands where the exponent of $ \omega$
is $\geq \eta $.  
Thus surely 
$ \alpha + \omega ^ \eta >
 \alpha \+ \beta  $,
with no need to compute
explicitly those summands which are
$<\omega ^ \eta$. 
 \end{proof} 

Parentheses are usually necessary in expressions involving
\emph{both} $+$ and $\+$; for example,
$( 1 \+ 0 ) + \omega = \omega \not= \omega + 1 = 1 \+ (0 + \omega )$,
or 
$(1+0) \+ \omega = \omega +1 \not= \omega = 1 + (0 \+ \omega )$.

\begin{definition} \labbel{natsumw}
Suppose that $( \alpha_i) _{i < \omega } $ 
is a countable sequence of ordinals, and set
$S_n= \alpha _0 \+ \dots \+ \alpha  _{n-1} $, for every 
$ n < \omega$. The 
\emph{natural sum} of $( \alpha_i) _{i < \omega } $ is
\begin{equation*} 
\nsum _{i < \omega } \alpha _i = \sup _{n < \omega } S_n 
 \end{equation*}    
The above natural sum is denoted
by $\sum^{\#} _{i < \omega } \alpha _i$ 
in \cite{VW}. 
 \end{definition}   

In the above notation $ \alpha _0 \+ \dots \+ \alpha  _{n-1} $
we conventionally allow $n=0$, and assume that 
$0$ is the outcome of such an ``empty'' sum. 
Notice that the notation is not ambiguous, in view of
Proposition \ref{facts}(1). 

\begin{proposition} \labbel{factsw}
Let $\alpha_i$, $\beta_i$ be ordinals and $ n,m < \omega$.  
  \begin{enumerate}[(1)]    
\item
$\sum_{i < \omega  } \alpha _i
\leq
\nsum _{i < \omega   } \alpha _i$
\item
If $\beta_i \leq \alpha _i$, for every $ i < \omega$,
then
$\nsum _{i < \omega   } \beta  _i \leq \nsum _{i < \omega   } \alpha _i$    
\item
If $n < m$, then
$S_n \leq S_m$; equality holds if and only if 
$\alpha_n= \dots= \alpha _{m-1} =0$.   
\item
$S_n  \leq \nsum _{i < \omega } \alpha _i$; equality holds
if and only if $\alpha_i=0$, for every $i \geq n$.  
\item
If  $\pi$ is a permutation of $ \omega$, then
$\nsum _{i < \omega } \alpha _i = \nsum _{i < \omega } \alpha _{\pi(i)}$
\item
More generally, suppose that $(F_h) _{h < \omega } $ is a partition
of $ \omega$ into finite subsets,
say, $F_h = \{ j_1, \dots, j _{r(h)} \} $, for every 
$ h \in \omega$.  Then
\begin{equation*}
\nsum _{i < \omega   } \alpha _i =
\nsum _{h < \omega   } \,\pl _{j \in F_h}  \alpha _j=
\nsum _{h < \omega   } ( \alpha _{j_1} \+ \alpha _{j_1} \+ \dots \alpha _{j _{r(h)}}  )
  \end{equation*}    
   \end{enumerate}
 \end{proposition} 

\begin{proof}
(1)-(4) are immediate from the definitions and 
Proposition \ref{facts} (1)-(2).

Clause (5) is a remark in the proof of 
\cite[Proposition 4.4]{VW}.
Anyway, (5) is the particular case  
of (6) when all the $F_h$'s are singletons.

To prove (6), define, for $ h < \omega$
\begin{equation*}
T _{h} =  \left( \pl _{j \in F_0}  \alpha _j \right)
\+
\dots
\+
 \left( \pl _{j \in F _{h-1} }  \alpha _j \right)
  \end{equation*}    
Thus the right-hand of the equation in (6)
is
$\sup _{h < \omega } T_h $.  
For 
$ h < \omega$,
let $m = \max _{0 \leq \ell < h} F_ \ell$.
The maximum exists since 
each $F_ \ell$ is finite, and we are 
considering only a finite number of 
 $F_ \ell$'s at a time.
Then each summand in the expansion of $T_h$
appears in $S _{m+1} $ (taking into account multiplicities), hence, by 
(4) and monotonicity
of the natural sum,
$\nsum _{i < \omega   } \alpha _i 
\geq 
S _{m+1}
\geq
T_h $.
Hence 
$ \nsum _{i < \omega   } \alpha _i \geq
\sup_{h < \omega } T_h $.
The reverse inequality is similar and easier. 
\end{proof}

The assumption that each $F_h$ is finite in condition (6)
above is necessary. For example, take
$\alpha_i = 1$, for every $ i <\omega$,
thus
$\nsum _{i < \omega } \alpha _i = \omega $. 
Suppose that there is some infinite $F_{\bar{h}}$.
Then 
$\nsum _{j \in F_{\bar{h}}} \alpha _j = \omega $. 
If $ \omega \setminus F_{\bar{h}} \not= \emptyset  $, then
 $\nsum _{h < \omega   } \left( \nsum_{j \in F_h}  \alpha _j \right)
\geq \omega \+ 1 > \omega $. 
 
Not everything from Proposition \ref{facts}
generalizes to infinite sums. For example, the operation
$\nsum _{i < \omega   } \alpha _i $, though monotone, 
as stated in (2) above,   is 
\emph{not} strictly monotone. 
E.~g.,  $\nsum _{i < \omega } 2
=
\nsum _{i < \omega } 1 = \omega $.
Actually,
 $\nsum _{i < \omega } \alpha _i = \omega $,
for every choice of the $\alpha_i$'s such that 
 $ \alpha _i < \omega $, for every $ i < \omega$,
and such that there are infinitely many nonzero $\alpha_i$'s. 

Condition (5) above can be interpreted as a 
version of commutativity,
and (6) as a version of 
the generalized commutative-associative law.
However, not all forms of associativity hold.
We have seen that we cannot associate infinitely many summands
inside some natural sum.
Similarly, we are not allowed to  ``associate inside out''.
Indeed, 
$ \omega +1=
1 \+ \nsum _{i < \omega } 1 \not= 
\nsum _{i < \omega } 1 = \omega $.
This is 
 a general and well-known fact.
For infinitary operations, 
some very weak form of generalized associativity
implies some form of absorption.

\begin{example} \labbel{notsimult}  
Suppose that $\oplus$ is a binary operation on some set $X$,
and $a \in X$ is such that $ a \oplus x \not= x$, for every $x \in X$.  
There is no infinitary operation $\bigoplus$ 
on $X$ such that 
\begin{equation*}\labbel{oplus}
x_0 \oplus \bigoplus _{i \in \omega } x _{i+1}=
  \bigoplus _{i \in \omega } x _{i}  
\end{equation*}    
for every sequence $(x_i) _{i \in \omega } $ 
of elements of $X$.
Indeed, taking $x_i= a$, for every $ i \in \omega$,
and letting 
$x=\bigoplus _{i \in \omega } x _{i}$,
we get   
$a \oplus x = x$, a contradiction. 
\end{example}

\section{Computing the infinite natural sum} \labbel{natsumwsec}

\begin{theorem} \labbel{propmain}
If $(\alpha_i) _{i < \omega} $ is a sequence of ordinals,
then there is $ m < \omega$ such that the following
hold, for every $n \geq m$.  
\begin{equation}\labbel{eq1}
\nsum _{n\leq i < \omega } \alpha _i 
=
 \sum _{n\leq i < \omega } \alpha _i 
 \end{equation}
\item
\begin{equation}\labbel{eq}
\begin{split}     
 \nsum _{ i < \omega } \alpha _i
& =
  ( \alpha _0 \+ \dots \+ \alpha _{n-1} )
+
  \nsum _{n\leq i < \omega } \alpha _i \\
& =
  ( \alpha _0 \+ \dots \+ \alpha _{n-1} )
+
  \sum _{n\leq i < \omega } \alpha _i 
\end{split} 
\end{equation} 
 \end{theorem} 

\begin{proof}
Let $\xi$ be the smallest ordinal such that 
the set $\{ i \in \omega \mid \alpha _i \geq \omega ^ \xi  \}$
is finite. 
Let $m$ be the smallest index such that 
$\alpha_i < \omega ^ \xi $,
for every $i \geq m$.   
The definition of $\xi$
 assures the existence of such an $m$. 
If $\xi=0$, then all but finitely many $\alpha_i$'s are $0$ 
and the proposition is trivial.

Suppose that 
$\xi$  is  a successor ordinal, say
 $\xi= \varepsilon +1$.
By the minimality of $\xi$, 
the set $\{ i \in \omega \mid \alpha _i \geq \omega ^ \varepsilon  \}$
is infinite, hence unbounded in $ \omega$.
Then $\nsum _{n\leq i < \omega } \alpha _i 
\geq
 \sum _{n\leq i < \omega } \alpha _i 
\geq 
\omega ^ \varepsilon \omega= \omega ^{\varepsilon+1} = \omega ^ \xi $.
Suppose that $\xi$ is limit.
By the definition of $\xi$,
we have that, for every $\varepsilon < \xi$,
there are infinitely many $i< \omega$   
such that $ \alpha _i \geq \omega ^ \varepsilon $. 
In particular, we can choose such an $i$ with $i \geq n$.
Then 
$ \sum _{n\leq i < \omega } \alpha _i \geq \alpha _i \geq \omega ^ \varepsilon $.
Since this holds for every $\varepsilon < \xi$,
we get  
$  \nsum _{n\leq i < \omega } \alpha _i \geq 
\sum _{n\leq i < \omega } \alpha _i \geq
\sup _{ \varepsilon < \xi}  \omega ^ \varepsilon =
\omega ^{\xi}$.
The inequality 
$  \nsum _{n\leq i < \omega } \alpha _i \geq 
\sum _{n\leq i < \omega } \alpha _i \geq \omega ^{\xi}$
is proved, no matter whether  $\xi$ is successor or limit. 

On the other hand, because of the definition of $m$,
if $i \geq n \geq m$, then $ \alpha _i < \omega ^ \xi$.   
By Proposition \ref{facts}(3),
$ \alpha _ n \+ \dots \+ \alpha _{\ell-1}< \omega ^ \xi$,
for every $ \ell \geq n$.
Hence
$\sum _{n\leq i < \omega } \alpha _i 
\leq
 \nsum _{n\leq i < \omega } \alpha _i 
=
\sup _{ \ell < \omega }  (\alpha _ n \+ \dots \+ \alpha _{\ell-1})
\leq \omega ^ \xi $.
In conclusion,
\begin{equation}\labbel{!}    
\nsum _{n\leq i < \omega } \alpha _i 
=
 \sum _{n\leq i < \omega } \alpha _i 
=
 \omega ^ \xi 
  \end{equation}
thus we have proved \eqref{eq1}. 

Let us now prove \eqref{eq}.
The inequality  
$ \nsum _{ i < \omega } \alpha _i
 \geq
  ( \alpha _0 \+ \dots \+ \alpha _{n-1} )
+
  \nsum _{n\leq i < \omega } \alpha _i$
 is trivial, since every ``partial sum'' on the right
is bounded by the partial sum on the left having the same length,
by Proposition \ref{facts}(2).
 For the other direction,
and recalling that 
$ S_ {\ell}$ denotes
$ \alpha _ 0\+ \dots \+ \alpha _{\ell-1}$,
observe that, by associativity, for every $\ell \geq n$,
we have 
$ S_ {\ell}=
S_n \+
 \alpha _ n \+ \dots \+ \alpha _{\ell-1}
<
S_n +
  \omega ^ \xi 
=
S_n + \nsum _{n\leq i < \omega } \alpha _i$,
where the strict inequality follows from repeated
applications of Proposition \ref{facts}(4), 
since 
$\alpha _ n , \dots, \alpha _{\ell-1}
<
  \omega ^ \xi $.
The last identity is from 
 equation \eqref{!}.
Since 
$\nsum _{ i < \omega } \alpha _i = \sup _{\ell < \omega } S_ {\ell}$
and since $S_ \ell$ is increasing, 
we get 
$\nsum _{ i < \omega } \alpha _i \leq S_n + \nsum _{n\leq i < \omega } \alpha _i$.

The identity
$\nsum _{ i < \omega } \alpha _i
 =
  ( \alpha _0 \+ \dots \+ \alpha _{n-1} )
+
  \sum _{n\leq i < \omega } \alpha _i$
is now immediate from 
\eqref{eq1}. It can be also proved in a way similar to above.
\end{proof}

Notice that the sum $+$ in equation \eqref{eq} 
cannot be replaced by a natural sum $\+$, that is,
we do not have, in general,
$ \nsum _{ i < \omega } \alpha _i
 =
S_n
\+
  \nsum _{n\leq i < \omega } \alpha _i $,
nor we have
$
\nsum _{ i < \omega } \alpha _i
 =
S_n
\+
  \sum _{n\leq i < \omega } \alpha _i 
$.
This is similar to the argument in Example \ref{notsimult}:
just take $\alpha_i =1$, for every $i \in I$;
then   
$\nsum _{ i < \omega } \alpha _i = \omega $
but
$ 
S_n
\+
  \nsum _{n\leq i < \omega } \alpha _i =
 S_n
\+
  \sum _{n\leq i < \omega } \alpha _i
 = n \+ \omega =  \omega +n$.
However, in
Corollary \ref{corexpl}
we shall show that the computation of a
countable natural sum can be actually reduced
to the computation of some finite natural sum.

\begin{remark} \labbel{bastinf}
Notice that equation \eqref{eq1} in Theorem \ref{propmain},
together with Proposition \ref{factsw}(5),
imply that
if $(\alpha_i) _{i < \omega} $ is a sequence of ordinals,
$m$ is given by Theorem \ref{propmain}, and $n \geq m$, then 
$
\sum _{n \leq i < \omega } \alpha _i 
=
 \nsum _{n \leq i < \omega } \alpha _{ i}  
=
 \sum _{n \leq i < \omega } \alpha _{ \pi (i)}$,
for every permutation $\pi$ of $[n, \omega)$. 
Actually, equation \eqref{!}  in the proof of Theorem \ref{propmain} shows that
it is enough to assume that
$\pi$ is a bijection
from 
$[n, \omega)$ to
$[n', \omega)$,
for some $n' \geq m$
(equation \eqref{!} does not hold if $\xi=0$, but this case is trivial). 

The result in the present remark
 can be obtained also as a consequence of a theorem 
by Sierpinski \cite{Sie}, asserting that a countable 
sum of nondecreasing ordinals is invariant under permutations.
Just notice that every sequence of ordinals is nondecreasing from some point on.
On the other hand,  Sierpinski's result is immediate from
equation \eqref{!}. 
Thus parts of the present note
can be seen as an extension of results from
\cite{Sie} to natural sums.
 \end{remark}

\section{Some kinds of mixed sums} \labbel{mixsec}

The definition of  a mixed sum of two ordinals
can be obviously extended to deal with 
infinitely many ordinals.

\begin{definition} \labbel{infmix}
Let $(\alpha_i) _{i \in I} $
be any sequence of ordinals (with no restriction on 
the cardinality of $I$). 
An ordinal $\gamma$ is a \emph{mixed sum}
of $(\alpha_i) _{i \in I} $  if there are 
pairwise disjoint
subsets $(A_i) _{i \in I} $ of $\gamma$
such that 
$\bigcup _{i \in I} A_i = \gamma  $
and, for every $i \in I$,
$A_i$ has order type $\alpha_i$,
with respect  to the order induced on $A_i$ 
by $\gamma$.    

In the above situation, we say
that $\gamma$ is a mixed sum of
$(\alpha_i) _{i \in I} $
\emph{realized by 
$(A_i) _{i \in I} $},
or simply that 
$(A_i) _{i \in I} $
is a \emph{realization of $\gamma$}.
Notice that $\alpha_i$ can be recovered by $A_i$,   
as embedded in $\gamma$.
 \end{definition}   

Notice that we could have given the above
definition just under the assumption that $\gamma$ and the
$\alpha_i$'s are linearly ordered sets,  not necessarily
well ordered. In this respect, notice that
any finite mixed sum of well ordered sets is itself necessarily 
well ordered; however, in case $I$ is infinite, 
the $\alpha_i$'s could ``mix themselves'' to a non well ordered set.
For example, starting with countably many $1$'s,
we could obtain \emph{every} countably infinite
 linear order as a mixed sum.
Throughout this note, however, 
and no matter how interesting the general case of linear orders is,
we shall always assume that $\gamma$ is an ordinal,
that is, well ordered.

\begin{theorem} \labbel{carr}
\emph{(Carruth \cite{Car}, Neumer \cite{neumer})} 
For every $ n < \omega$ and ordinal numbers
 $\alpha_0, \dots, \alpha _n$,  
the largest mixed
sum of $( \alpha_i) _{i \leq n} $ exists and is 
$\alpha_0 \+ \alpha _1 \+ \dots \+ \alpha _n $.
 \end{theorem}

As we hinted in the introduction,
and contrary to the finite case,
the set of all the mixed sums of an infinite sequence of ordinals
need not have a maximum. 
If we take 
$\alpha_i = 1$ for every $i< \omega $,
then every infinite countable ordinal  
is a mixed sum of 
$(\alpha_i) _{i \in \omega } $, thus the supremum of all the mixed sums
of $(\alpha_i) _{i \in  \omega } $
is $ \omega_1$,
which is \emph{not} a mixed sum of $(\alpha_i) _{i \in \omega } $.
Hence there is some interest in restricting ourselves 
to well-behaved mixed sums

\begin{definition} \labbel{wellmixed} 
We say that 
$\gamma$ is a \emph{left-finite} mixed sum of
$(\alpha_i) _{i \in I} $
if $\gamma$ can be realized as a mixed sum by
$(A_i) _{i \in I} $ in such a way that,
for every $ \delta  <\gamma $, 
the set
$\{ i \in I \mid A_i \cap \delta    \not= \emptyset   \}$ 
is finite;
in words, for every $ \delta  < \gamma $, the predecessors of $\delta$ are all 
taken from finitely many $A_i$'s. 

Given a realization $(A_i) _{i \in I} $ of $\gamma$ 
and $i \in I$, 
we say that
$A_i$ is \emph{convex (in $\gamma$)} if
 $[a, a']_ \gamma  = \{ \delta \in \gamma  \mid a \leq \delta \leq a' \}\subseteq A_i$, whenever  $a< a' \in A_i$.

We say that 
 $\gamma$ is a \emph{piecewise convex}
 (resp., an \emph{almost piecewise convex})
mixed sum of $(\alpha_i) _{i \in I} $
if $\gamma$  can be realized  in such a way that
all  the $A_i$'s 
(resp., all but  a finite number of the $A_i$'s)
are convex in $\gamma$.
For brevity, we shall write 
\emph{pw-convex} in place of
  piecewise convex.
\end{definition}   

If
 $\gamma$ is a pw-convex
mixed sum of $(\alpha_i) _{i \in I} $,
as realized by $(A_i) _{i \in I} $,
then, for every $i \not= j \in I$ and 
$\delta, \varepsilon  \in A_i$,  
$ \delta ' , \varepsilon' \in A_j$,
we have that
$\delta < \delta ' $ 
if and only if 
$ \varepsilon  < \varepsilon '$.
In this way, if each
$A_i$ is nonempty, the order on $\gamma$ 
induces an order (in fact, a well order)
on $I$.  
Hence we can reindex 
$( A_i) _{i \in I} $ 
as
$( A_{\pi(\iota)}) _{\iota < \theta } $
for some ordinal $\theta$ 
and some bijection
$\pi: \theta \to I$
in such a way that
 $\delta < \delta ' $,
whenever
$\delta \in A _{\pi(\iota)} $,
$ \delta '  \in A _{\pi(\iota')} $
and
$ \iota < \iota'$. 
Then an easy induction shows that
$\gamma= \sum _{\iota < \theta} \alpha _{\pi(\iota)} $. 
If in addition $\gamma$ is left finite, then
necessarily $\theta \leq \omega $.

Conversely, if $\gamma= \sum _{\iota < \theta} \alpha _{\pi(\iota)}$,
for some reindexing of the $\alpha_i$'s,  
then trivially $\gamma$ is a pw-convex
mixed sum of $(\alpha_i) _{i \in I} $,
and if $\theta \leq \omega $,
then $\gamma$ is also left finite. 
We have proved the next proposition. 

\begin{proposition} \labbel{sum}
 Suppose that $( \alpha_i) _{i \in I} $ is a sequence of ordinals,
and $\alpha_i > 0$, for every $i \in I$.
Then $\gamma$ is a pw-convex (pw-convex  and left-finite) mixed sum of
$( \alpha_i) _{i \in I} $ if and only if there are
some ordinal $\theta$ (with $ \theta \leq \omega $)
and a bijection 
$\pi: \theta \to I$ such that 
$ \gamma = \sum _{\iota < \theta} \alpha _{\pi( \iota)} $.   
\end{proposition} 

\begin{remark} \labbel{lf}
There might be infinitely many
left-finite mixed sums of the same sequence.
Indeed, take $\alpha_i = \omega $, for every $i < \omega$.
Since $ \omega$ is the union of countably many disjoint countably
infinite sets,
we see that $ \omega$ is a (necessarily left-finite) mixed sum
of $( \alpha_i) _{i < \omega } $.
By moving just one copy of $ \omega$ ``to the bottom''
we get that also $ \omega+ \omega $
 is a left-finite mixed sum
of $( \alpha_i) _{i < \omega } $.  
Iterating, for every $ n < \omega$ 
we get $ \omega n$ as a left-finite mixed sum
of $( \alpha_i) _{i < \omega } $. 
Also $ \omega^ 2$ 
is a left-finite mixed sum
of $( \alpha_i) _{i < \omega } $;
by Proposition \ref{sum} 
it is the only one which is left-finite and pw-convex;
actually, it is the only one which is left-finite and almost pw-convex.
 \end{remark}   

\begin{theorem} \labbel{main}
If $(\alpha_i) _{i < \omega} $ is a sequence of ordinals, 
then $\nsum _{i < \omega } \alpha _i$ 
is a mixed sum of
$(\alpha_i) _{i < \omega} $.
In fact, $\nsum _{i < \omega } \alpha _i$
is the largest left-finite mixed sum
of $(\alpha_i) _{i < \omega} $, and also
the largest left-finite  and almost pw-convex
mixed sum of $(\alpha_i) _{i < \omega } $.
 \end{theorem} 

\begin{proof} 
By equation \eqref{eq} in  
Theorem \ref{propmain}, 
we have
$\nsum _{i < \omega } \alpha _i =
(\alpha _0 \+ \dots \+ \alpha _{n-1}) + \sum _{n \leq i < \omega } \alpha _i$,
for some $ n < \omega$.
By Theorem \ref{carr}, 
$ \gamma _1 =\alpha _0 \+ \dots \+ \alpha _{n-1}$
is a mixed sum of
$ \alpha _0, \dots,  \alpha _{n-1}$.
By the easy part of Proposition \ref{sum},
$ \gamma _2= \sum _{n \leq i < \omega } \alpha _i $ 
is a left-finite pw-convex mixed sum
of $( \alpha_i) _{n \leq i < \omega } $.
Putting   the members of the realization of
$\gamma_1$ at the bottom, and     
 the members of the realization of
$\gamma_2$ at the top,
we realize
$\nsum _{i < \omega } \alpha _i =
 \gamma _1 + \gamma _2$
as a 
left-finite and almost pw-convex mixed sum
of $( \alpha_i) _{ i < \omega } $.

To finish the proof of the  theorem it is enough to show that if
$\gamma$ is any left-finite mixed sum
of $(\alpha_i) _{i < \omega} $, then 
$\gamma \leq \nsum _{i < \omega } \alpha _i$. 
Let $\gamma$ be a left-finite mixed sum
of $(\alpha_i) _{i < \omega} $ as realized by
$(A_i) _{i < \omega} $.
If all but a finite number of the $\alpha_i$'s are $0$,
then the result is immediate from Theorem \ref{carr}. Otherwise, 
left finiteness implies that $\gamma$ is a limit ordinal.
If $\gamma' < \gamma $, then
 $( \gamma' \cap A_i) _{i < \omega} $
witnesses that $\gamma'$
is a mixed sum of $( \beta _i) _{i \in I} $,
where, for every $ i < \omega$,  $\beta_i$ is the order type of 
$\gamma \cap A_i$; thus $\beta_i \leq \alpha _i$.    
Left finiteness implies that only a finite number
of the $\beta_i$'s are nonzero, thus,
again by Theorem \ref{carr}, 
$\gamma' \leq \beta _{i_1} \+ \dots \+  \beta _{i_ \ell}$,
for certain distinct indices $i_1, \dots , i_ \ell$.
 Taking $n = \sup \{ i_1, \dots , i_ \ell\} $, we get 
$\gamma' \leq \beta _{i_1} \+ \dots \+  \beta _{i_ \ell}
\leq 
 \alpha  _{i_1} \+ \dots \+  \alpha  _{i_ \ell}
\leq 
\alpha _0\+ \dots \+ \alpha _n
<
 \nsum _{i < \omega } \alpha _i
$.
Since $\gamma$ is limit
and $\gamma' \leq  \nsum _{i < \omega } \alpha _i$,
for every $\gamma' < \gamma $,
we get   $\gamma \leq  \nsum _{i < \omega } \alpha _i$.
\end{proof}

\section{Expressing sums in terms of the normal form} \labbel{cantsec}

The proof  of
Theorem  \ref{propmain}
gives slightly more.
Let $\alpha$ and 
$\xi$ be ordinals,
and express $\alpha$ in Cantor normal 
form as
$
 \omega ^ {\eta_k} r_k + 
\dots
+ \omega ^ {\eta_0}r_0$.
The ordinal $\alpha^{\restriction  \xi}$, in words,  
\emph{$\alpha$ truncated at the $\xi^{th}$ exponent of $ \omega$}, 
is
$
 \omega ^ {\eta_k} r_k + 
\dots
+ \omega ^ {\eta_ \ell}r_\ell$,
where $\ell$ is the smallest index
such that $\ell \geq \xi$.   
The above definition should be intended in the sense that
$\alpha^{\restriction  \xi}=0$
in case that $\alpha < \omega ^ \xi$.  

\begin{corollary} \labbel{corexpl}
Suppose that $( \alpha_i) _{i < \omega } $ 
is a sequence of ordinals, and
let $ \xi$ be the smallest ordinal
such that $\{ i < \omega \mid \alpha _i \geq \omega ^\xi\}$
is finite.
Enumerate those $\alpha_i$'s such that 
 $\alpha _i \geq \omega ^\xi$ as
$ \alpha _{i_0}, \dots, \alpha _{i_h}  $,
with $i_0 < \dots < i_h$
(the sequence might be empty).
If $\xi>0$, then

\begin{align} \labbel{trunc}
 \nsum _{ i < \omega } \alpha _i
& =
  ( \alpha _{i_0} \+ \dots \+ \alpha _{i_h}) 
+ \omega ^\xi 
=
 \alpha _{i_0}^{\restriction  \xi} \+ \dots \+ \alpha _{i_h}^{\restriction  \xi}
\+ \omega ^\xi \ \text{ and } 
 \\
\labbel{truncsum}
 \sum _{ i < \omega } \alpha _i
& =
   \alpha _{i_0} + \dots + \alpha _{i_h} 
+ \omega ^\xi =
   \alpha _{i_0}^{\restriction  \xi} + \dots + \alpha _{i_h}^{\restriction  \xi} 
+ \omega ^\xi\, ;
\end{align}
moreover, for every $\varepsilon< \xi$,
we have
\begin{equation*}\labbel{ult}    
 \nsum _{ i < \omega } \alpha _i=
 \nsum _{ i < \omega } \alpha _i^{\restriction  \varepsilon }
\qquad \text{  and  } \qquad
 \sum _{ i < \omega } \alpha _i=
 \sum _{ i < \omega } \alpha _i^{\restriction  \varepsilon }
  \end{equation*}  
\end{corollary} 

\begin{proof}
The $\xi$ defined in the statement of the present corollary 
is the same as the $\xi$ defined in the proof 
of Theorem \ref{propmain};
and  the $\alpha _{i_h}$ defined here
is the same as $a _{m-1} $ in that proof
(if the sequence of the
$\alpha _{i_\ell}$'s is not empty).
Equation \eqref{!} in the proof of   Theorem \ref{propmain} 
gives 
$\nsum _{m\leq i < \omega } \alpha _i 
=
 \omega ^ \xi $.
By commutativity and associativity of $\+$,
and using 
Proposition \ref{facts}(5),
 equation \eqref{eq} 
in Theorem  \ref{propmain}
becomes exactly the first identity in
equation \eqref{trunc}. 
The second identity is easy ordinal 
arithmetic, noticing that
$\alpha+ \omega ^ \xi = \alpha ^{\restriction  \xi}
+ \omega ^ \xi$ and
$( \alpha \+ \beta )^{\restriction  \xi} =
\alpha^{\restriction  \xi} \+ \beta ^{\restriction  \xi}$,
for every $\alpha$ and $\beta$.  

The proof of  \eqref{truncsum} is similar, using the
fact that
$\sum _{ i < \omega } \alpha _i
 =
   \alpha _{0} + \dots + \alpha _{m-1} + 
\sum _{ m \leq i < \omega } \alpha _i
$.
Then one should use the identity
$ \beta + \gamma + \omega ^ \xi =
 \gamma  + \omega ^ \xi$,
holding whenever 
$ \beta   < \omega ^ \xi $.
Indeed, if $\gamma < \omega ^ \xi$,
then all sides are equal
to $\omega ^ \xi$; otherwise, if $\gamma \geq \omega ^ \xi$,
then $\beta$ is absorbed by $\gamma$, since it is already absorbed
by the leading term in the Cantor normal expression of $\gamma$.
See \cite{Sie}.

To prove the last two identities, 
notice that if
$\varepsilon < \xi$, then 
$\xi$ is also the least 
 ordinal
such that 
$\{ i < \omega \mid \alpha _i^{\restriction  \varepsilon } \geq \omega ^\xi\}$
is finite.
Hence we can apply 
 \eqref{trunc} twice to get
$\nsum _{ i < \omega } \alpha _i^{\restriction  \varepsilon }
=
( \alpha _{i_0}^{\restriction  \varepsilon })^{\restriction  \xi}
 \+ \dots \+
( \alpha _{i_h}^{\restriction  \varepsilon })^{\restriction  \xi}
\+ \omega ^\xi=
 \alpha _{i_0}^{\restriction  \xi} \+ \dots \+ \alpha _{i_h}^{\restriction  \xi}
\+ \omega ^\xi
=
 \nsum _{ i < \omega } \alpha _i$. 
The last identity is proved in the same way,
using equation \eqref{truncsum}.
\end{proof}

Notice that Corollary \ref{corexpl} 
furnishes a method to 
compute $\nsum _{i < \omega } \alpha _i$
and $\sum _{i < \omega } \alpha _i$
in terms of the Cantor normal forms of
the $\alpha_i$'s,
in fact, of just finitely many 
$\alpha_i$'s, once $\xi$ has been determined. 

One cannot expect 
that, for every sequence $( \alpha_i) _{i \in \omega } $
of ordinals, there is some permutation of $ \omega$
such that 
$\nsum _{ i < \omega } \alpha _i 
=
 \sum _{ i < \omega } \alpha _{\pi(i)}
$.
The counterexample has little to do with infinity:
just take two ordinals $\alpha_0$ and 
$\alpha_1$
such that 
$\alpha_0  \+ \alpha_1 \not= \alpha_0  + \alpha_1 $
and 
$\alpha_0  \+ \alpha_1 \not= \alpha_1  + \alpha_0 $,
for example, 
$\alpha_0= \alpha_1= \omega +1$.
Then, setting 
$\alpha_i = 0$, for $i>1$,
we have
$\alpha_0  \+ \alpha_1  = \nsum _{ i < \omega } \alpha _i 
\not=
 \sum _{ i < \omega } \alpha _{\pi(i)}$, for every permutation $\pi$.
Of course, we can arrange things in order to have some really infinite sum, e.~g.,
take $\alpha_0= \alpha_1= \omega^2 + \omega $ and $\alpha_i=1$, for $i \geq 2$.  

\begin{lemma} \labbel{lemult}
Suppose that $\gamma$ is a mixed sum
of $\alpha_1$  and $\alpha_2$ with $\alpha_1 < \alpha _2= \omega ^\xi$,
for some $\xi >0$, and the mixed sum is realized by
$A_1$, $A_2$ in such a way that $A_2$
is cofinal in $\gamma$. Then $\gamma=\omega ^\xi$.
 \end{lemma} 

\begin{proof}
$\gamma \geq \alpha _2= \omega ^\xi$ is trivial.
For the other direction, 
let $\delta < \gamma $.
Thus $\delta$ is a mixed sum of $\beta_1$,
$\beta_2$, where    
$\beta_1$,
$\beta_2$ are the order types of, respectively,
$A_1 \cap  \delta $, $A_2 \cap \delta $.
Since  $\delta< \gamma $, $A_2$ is cofinal in $\gamma$ and
$A_2$ has   order type $ \omega^\xi$,
then  $A_2 \cap \delta $ has order type 
$ < \omega^\xi$.
Moreover 
$A_1 \cap \delta $ has order type 
$ \leq \alpha _1< \omega^\xi$.
Since $\delta$ is a mixed sum of 
$\beta_1$ and $\beta_2$,
then, by Theorem \ref{carr},
$ \delta \leq \beta _1 \+ \beta _2 $.
We have showed that 
$\beta_1, \beta _2 < \omega ^ \xi$,
hence  $ \delta \leq \beta_1 \+ \beta _2 < \omega ^ \xi$,
by Proposition \ref{facts}(3). 
Since $\delta < \omega ^\xi$,
for every $\delta< \gamma $,
then $\gamma \leq \omega ^\xi$.   
 \end{proof}

\begin{theorem} \labbel{cor}
If $( \alpha_i) _{i \in I} $ is a sequence of ordinals,
then there are at most a finite number
of left-finite and almost pw-convex
mixed sums of $( \alpha_i) _{i \in I} $.
 \end{theorem}

\begin{proof} 
If there is some
left-finite 
mixed sum of $( \alpha_i) _{i \in I} $,
then necessarily all but countably many $\alpha_i$'s
are $0$.
Hence it is no loss of generality to assume that 
$|I| \leq \omega $.
If all but a finite number of the $\alpha_i$'s are $0$,
then the corollary follows from a theorem 
by L{\"a}uchli \cite{lauchli},
asserting that a finite set of ordinals
has only a finite number of mixed sums.
In conclusion, we can assume that  
$I= \omega $ and
$\alpha_i\not=0$, for every $ i < \omega$. 

Suppose that 
$\gamma$ is a 
left-finite and almost pw-convex
mixed sum of $( \alpha_i) _{i < \omega } $,
as realized by $( A_i) _{i < \omega } $.
Assume the notation in Corollary \ref{corexpl}. 
We shall show that 
$\gamma$ is a mixed sum of 
$ \alpha _{i_0}, \dots,  \alpha _{i_h}, \omega ^\xi$. 
This will give the result
by the mentioned theorem from L{\"a}uchli \cite{lauchli}.

Since 
$\gamma$ is realized as an 
almost pw-convex
mixed sum,
there is $n$ such that 
$A_i$ is convex, for every $i \geq n$.
Without loss of generality,
choose $n>i_h$, or, which is the same, 
$n \geq m$, where $m$ is given by
Theorem \ref{propmain}.
Let 
$A= \bigcup _{i \geq n} A_i $,
thus $\gamma$ is realized as a finite mixed sum
by $A_0, \dots, A _{n-1}, A $.
Moreover, by the left finiteness of the realization $( A_i) _{i \in I} $,
and since we have assumed that the $A_i$'s are nonempty, 
we get that $A$ is cofinal in $\gamma$. 
The order type of $A$ 
is $\sum _{n \leq i < \omega } \alpha _{ \pi(i)}   $,
for some permutation of $\pi$ of $[n, \omega )$,
by Proposition \ref{sum},
since $A= \bigcup _{i \geq n} A_i $,
 the $A_i$'s are convex
(in $\gamma$, hence, a fortiori, in $A$), for $i \geq n$,
and the realization 
$( A_i) _{i < \omega } $
is left-finite.
By Remark \ref{bastinf} and equation \eqref{!} in the proof of 
Theorem \ref{propmain},
we get that $A$ has order type $ \omega^ \xi$.    
Now let 
$F= \{ 0, \dots, n-1\} \setminus \{ i_0, \dots, i_h\}  $, and set
$A' = A  \cup \bigcup _{j \in F} A_j$.
By repeated applications of Lemma \ref{lemult},
we get that $A'$ has order type $ \omega^ \xi$.
Since $\gamma= A' \cup A _{ i_0} \cup \dots \cup A _{ i_h} $,
then $\gamma$ is a mixed sum of 
$\omega ^\xi, \alpha _{i_0}, \dots,  \alpha _{i_h}$,
what we wanted to prove. 
\end{proof}  

In view of
Proposition \ref{sum}, 
Theorem \ref{cor} extends a classical result
by  Sierpinski \cite{Sie}, 
asserting that
$\sum _{i < \omega } \alpha _{\pi(i)}$
assumes only finitely many values,
$\pi$ varying among all permutations of $ \omega$.

By Remark \ref{lf}, the assumption of 
almost pw-convexity  is necessary in 
Theorem \ref{cor}. 

\begin{remark} \labbel{shift}
A \emph{shifted sum}
is defined in the same way as a mixed sum, except
that we do not require 
the $A_i$'s to be pairwise disjoint.
See \cite{OC} for applications  of finite shifted sums.  
Since every mixed sum  is a shifted sum and,
on the other hand,
given a shifted sum  of $( \alpha_i) _{i \in I} $,
there is always some larger mixed sum of 
$( \alpha_i) _{i \in I} $, we get that
Theorem \ref{carr} 
holds for shifted sums in place of mixed sums.
Then the proof of Theorem \ref{main}, too,
carries over to get the result for shifted sums in place of mixed sums. 
On the other hand, the analogue of Theorem \ref{cor}
does not hold for shifted sums. Indeed, 
let $\alpha_i = \omega $, for $i <  \omega$. 
Then $ \omega$ can be realized as
a left-finite pw-convex shifted sum of $( \alpha_i) _{i < \omega } $
by $ A_i= [i, \omega )$.
Then we can get infinitely many    
left-finite pw-convex shifted sums of    $( \alpha_i) _{i < \omega } $
arguing as in  Remark \ref{lf}.
\end{remark}

\begin{remark} \labbel{surreal}
It is clear that, when restricted to the class of ordinals,
the surreal number addition is the natural sum.
See Alling \cite{All}, Ehrlich \cite{E} and
Gonshor \cite{Go}  for information about the surreal numbers.
Corollary \ref{corexpl}
suggests the possibility of extending the  infinitary natural sum
 to the class of those surreal numbers
which have positive coefficients 
in their Conway normal representation.
Recall that every surreal number $x$ can be uniquely
expressed in \emph{Conway normal form} as
$x= \sum _{s \in S} \omega ^s r_s $,
where $S$ is a reverse-well-ordered 
set of surreal numbers and the
$r_s$ are nonzero real numbers.
In case $x$ is an ordinal
 the Conway and the Cantor normal forms
coincide. 
 
Let 
$( x_i) _{i < \omega } $ 
be a countable sequence of surreal numbers
with normal forms
$x_i= \sum _{s \in S_i} \omega ^s r _{s,i}  $
and such that all the $r _{s,i}$'s are positive
(we shall show later that this request can be somewhat weakened).
Let 
$S= \bigcup _{i < \omega } S_i$
and define a subset $S^*$
of $S$ by declaring
$s \in S^* $ if and only if 
$\{ t \in S \mid t \geq s \} $ is reverse-well-ordered.
It might well happen that $S^* = \emptyset $; however,
in any case, $S^*$ is  reverse-well-ordered.
For
$s\in S$,
let 
$c_s = \sum _i  r _{s,i}$,
where the sum is taken among all 
$i< \omega $ such that  $s \in S_i$.
This might be either a finite sum or a countably infinite sum of 
positive real numbers;
in the latter case we consider it as an infinite sum in
the sense of classical analysis.
We allow the possibility $c_s= \infty $, that is, $ c_s= \omega$ 
in the surreal sense.

Suppose first that
$S=S^*$.
In this case we define
\begin{equation}\labbel{sugs}     
\nsum _{i < \omega } x_i =  \sum _{s \in S} \omega ^s c_s
 \end{equation} 
This is a well-defined surreal number,
since $S=S^*$, hence $S$ is reverse-well-ordered.
Strictly speaking, equation \eqref{sugs}
does not necessarily give a 
normal form representation, due to the possibility that 
some $c_s$ is $ \omega$.
Formally,
$\nsum _{i < \omega } x_i =  \sum _{s \in T} \omega ^s c_s + \sum _{s \in U} \omega ^{s+1}$, where 
$U = \{ s \in S \mid c_s = \omega \} $ and
$T= S \setminus U$.

Notice also that, in order to give the definition in
equation \eqref{sugs} we only need that
$S=S^*$ and that the sums 
$\sum _{i}  r _{s,i} $ 
are well-defined in the sense of classical analysis,
no matter whether all the $r _{s,i} $ are positive.

Now suppose that
$S\not=S^*$ and let
$\bar{s} $ be the surreal number
$ \{ S \setminus S^* \mid \emptyset  \} $.
In this case we define
\begin{equation}\labbel{snugs}     
\nsum _{i < \omega } x_i = \omega ^{\bar{s}}+  \sum _{s \in S^*} \omega ^s c_s
 \end{equation} 
The definition in equation \eqref{snugs} 
makes sense just under the assumption that
$c_s =\sum _{i}  r _{s,i} $ 
is well-defined,
for every $s \in S^*$.
However, it is natural to ask that
there is  $s \in S \setminus S^*$ such that 
$\sum _{i}  r _{t,i} $ 
is well-defined \emph{and strictly positive},
for every  $t \in S \setminus S^*$ with
$t \geq s$. In case that there is 
$s \in S \setminus S^*$ such that 
$\sum _{i}  r _{t,i} $ 
is well-defined and strictly negative,
for every  $t \in S \setminus S^*$ with
$t \geq s$, a more natural definition would be
$\nsum _{i < \omega } x_i = - \omega ^{\bar{s}}+  \sum _{s \in S^*} \omega ^s c_s$.

In the special case when each $x_i$ is an ordinal, 
the definitions given by
\eqref{sugs} and \eqref{snugs} 
coincide with Definition \ref{natsumw}, by
Corollary \ref{corexpl}.

The definition given in the present remark is quite tentative, 
and is not the only possible one.
Notice that the definition of
$\nsum _{i < \omega } x_i$ given here 
for surreal numbers does not satisfy the analogue of 
Proposition \ref{factsw}(2).
Indeed, let 
$x_i=  \omega^i$   
and 
$y_i=  \omega ^{ \omega -2} $,
for 
$i< \omega$.
Then 
$x_i < y_i$,
for every $i< \omega$; 
 however,
$ \nsum _{i < \omega } x_i =
\omega ^ \omega >
\omega ^{ \omega -1}
=
 \omega ^{ \omega -2} \omega =
 \nsum _{i < \omega } y_i
 $. 
 \end{remark}

\section*{Disclaimer} \labbel{disc} 
Though the author has done his best efforts to compile the following
list of references in the most accurate way,
 he acknowledges that the list might 
turn out to be incomplete
or partially inaccurate, possibly for reasons not depending on him.
It is not intended that each work in the list
has given equally significant contributions to the discipline.
Henceforth the author disagrees with the use of the list
(even in aggregate forms in combination with similar lists)
in order to determine rankings or other indicators of, e.~g., journals, individuals or
institutions. In particular, the author 
 considers that it is highly  inappropriate, 
and strongly discourages, the use 
(even in partial, preliminary or auxiliary forms)
of indicators extracted from the list in decisions about individuals (especially, job opportunities, career progressions etc.), attributions of funds, and selections or evaluations of research projects.

\end{document}